\documentclass[12pt]{amsart}
\usepackage{fullpage, amsmath}
\usepackage{url}

\newtheorem{theorem}{Theorem}
\newtheorem{corollary}[theorem]{Corollary}

\theoremstyle{remark}

\theoremstyle{definition}

\begin{document}

\title{Binary words avoiding $xx^Rx$ and strongly unimodal sequences}

\author{James Currie and Narad Rampersad}

\address{Department of Mathematics and Statistics \\
University of Winnipeg \\
515 Portage Avenue \\
Winnipeg, Manitoba R3B 2E9 (Canada)}

\email{j.currie@uwinnipeg.ca, n.rampersad@uwinnipeg.ca}


\date{\today}

\begin{abstract}
In previous work, Currie and Rampersad showed that the growth of the
number of binary words avoiding the pattern $xxx^R$ was intermediate
between polynomial and exponential.  We now show that the same holds
for the growth of the number of binary words avoiding the pattern
$xx^Rx$.  Curiously, the analysis for $xx^Rx$ is much simpler than
that for $xxx^R$.  We derive our results by giving a bijection between
the set of binary words avoiding $xx^Rx$ and a class of sequences
closely related to the class of ``strongly unimodal sequences.''
\end{abstract}

\maketitle

\section{Introduction}
In this paper we give an exact characterization of the binary words
avoiding the pattern $xx^Rx$.  Here the notation $x^R$ denotes the
\emph{reversal} or \emph{mirror image} of $x$.  For example, the word
$0101\,1010\,0101$ is an instance of $xx^Rx$, with $x=0101$.  The set
of binary words avoiding the related pattern $xxx^R$ has been the
subject of recent study.  This study began with the work of Du,
Mousavi, Schaeffer, and Shallit \cite{DMSS14}, who observed that there
exist infinite periodic binary words that avoid $xxx^R$ and provided
an example of an aperiodic infinite binary word that avoids $xxx^R$.
Answering a question of Du et al., the present authors derived upper
and lower bounds for the number of binary words of length $n$ that
avoid $xxx^R$ and showed that the growth of this quantity was neither
polynomial nor exponential \cite{CR15}.  This result was the first
time such an intermediate growth rate had been shown in the context of
pattern avoidance.

In the present work we analyze the binary words avoiding $xx^Rx$.
Once more, we are able to show a growth rate for the number of such words
that is neither polynomial nor exponential.  Surprisingly, the
analysis is much simpler than what was required to show the analogous
result for the pattern $xxx^R$.  Indeed we are able to obtain a
complete characterization of the binary words avoiding $xx^Rx$ by
describing a correspondence between such words and sequences of
integers that are very closely related to \emph{strongly unimodal
  sequences} of integers.  These latter have recently been studied by
Rhoades \cite{Rho14}.  This correspondence provides a rather pleasing
connection between avoidability in words and the classical theory of
partitions.  Finally, we conclude this work with some remarks
concerning the non-context-freeness of the languages of binary words
that avoid the patterns $xxx^R$ and $xx^Rx$ respectively.

\section{Enumerating binary words avoiding $xx^Rx$}

Let
\[
L = \{w \in \{0,1\}^* : w \text{ avoids } xx^Rx\},
\]
and let $L = L_0 \cup L_1 $, where $L_0$ (resp.\ $L_1$) consists of
the words in $L$ that begin with $0$ (resp.\ $1$), along with the
empty word.

Note that the words in $L$ avoid both $000$ and $111$.
Observe that any $w \in \{0,1\}^*$ that avoids $000$ and $111$ has a
unique representation of the form
\begin{equation}\label{factorization}
w = A_0a_1a_1A_1a_2a_2A_2 \cdots A_{k-1}a_ka_kA_k,
\end{equation}
where each $A_i$ is a prefix (possibly empty) of either $010101\cdots$
or $101010\cdots$ and each $a_i \in \{0,1\}$.  Given such a
factorization, we define an associated sequence $f(w) =
(n_0,\ldots,n_k)$, where
\begin{itemize}
\item $n_0 = |A_0a_1|$,
\item $n_i = |a_iA_ia_{i+1}|$, for $0<i<k$, and
\item $n_k = |a_kA_k|$.
\end{itemize}

Let $X$ denote the set of all sequences of the form
$(d_1,d_2,\ldots,d_m)$ such that for some $j$, either
\begin{align*}
\text{(Type 1)} && 0 < d_1 < \cdots < d_{j-1} &< d_j > d_{j+1} > \cdots > d_m > 0,
\text{ or}\\
\text{(Type 2)} && 0 < d_1 < \cdots < d_{j-1} &= d_j > d_{j+1} > \cdots > d_m > 0.
\end{align*}
The \emph{weight} of any such sequence is the sum $d_1 +\cdots + d_m$.  We
also include the empty sequence in $X$.  Type~1 sequences
are called \emph{strongly unimodal sequences}.
Note that the set $X$ can also be equivalently defined as follows:
$X$ consists of all sequences $(d_1,d_2,\ldots,d_m)$ for which there
is no $j$ such that both $d_{j-1} \geq d_j$ and $d_j \leq d_{j+1}$.

We show the following.

\begin{theorem}\label{bijection}
The map $f$ defines a one-to-one correspondence between the words in
$L_0$ of length $n$ and the sequences in $X$ of weight $n$.
\end{theorem}

\begin{proof}
Let $w\in\{0,1\}^*$ be a word starting with $0$.  Let
$f(w) = (n_0,\ldots,n_k)$ and let
\[
w = A_0a_1a_1A_1a_2a_2A_2 \cdots A_{k-1}a_ka_kA_k,
\]
be the factorization given in \eqref{factorization}.

We first show that if $f(w) \notin X$ then $w \notin L_0$.  Since
$f(w) \notin X$ there is some $j$ such that both $n_{j-1} \geq n_j$
and $n_j \leq n_{j+1}$.  Define $B_1$, $B_2$, and $B_3$ as follows:
\begin{itemize}
\item if $j=1$ then $B_1$ is the suffix of $A_0a_1$ of length
  $n_1$; otherwise, $B_1$ is the suffix of $a_{j-1}A_{j-1}a_j$ of length
  $n_j$;
\item $B_2 = a_jA_ja_{j+1}$;
\item if $j=k-1$ then $B_3$ is the prefix of $a_kA_k$ of length
  $n_{k-1}$; otherwise $B_3$ is the prefix of $a_{j+1}A_{j+1}a_{j+2}$ of length
  $n_j$.
\end{itemize}
The conditions on $n_{j-1}$, $n_j$, and $n_{j+1}$ ensure that $B_1$,
$B_2$, and $B_3$ can be defined.  However, we now see that $B_1 =
B_2^R = B_3$, where $B_1$ is either $(01)^{n_j/2}$ or $(10)^{n_j/2}$.
The word $w$ thus contains an instance of $xx^Rx$, and
hence is not in $L_0$.

Next we show that if $w \notin L_0$, then $f(w) \notin X$.  First note
that if $w$ has a factor $w'$ such that $f(w') \notin X$, then $f(w)
\notin X$.  We may therefore suppose that $w = vv^Rv$ and contains no
smaller instance of the pattern $xx^Rx$.  Then there are indices $i<j$
such that
\begin{itemize}
\item $v = A_0a_1 \cdots A_{i-1}a_i$,
\item $v^R = a_iA_i \cdots A_{j-1}a_j$, and
\item $v = a_jA_{j+1} \cdots a_kA_k$.
\end{itemize}
If $k=2$ we necessarily have $f(w) = (n_0,n_0,n_0) \notin X$, so
suppose $k>2$.  Since $w$ contains no smaller instance of $xx^Rx$, we
must have $n_0 < n_1$.  However, we also have $n_{j-1} = n_j = n_0$
and $n_{j+1} = n_1$.  Thus, we have $n_{j-1} = n_j < n_{j+1}$ and so
$f(w) \notin X$.
\end{proof}

We now turn to the enumeration of the sequences in $X$.  Let $v(n)$
denote the number of sequences in $X$ of weight $n$. A sequence
$(d_1,d_2,\ldots,d_m)$ in $X$ can be represented by a pair
$(\lambda, \mu)$ of partitions into distinct parts, where the
partition $\lambda$ gives the increasing part of the sequence and the
partition $\mu$, read in the reverse order, gives the decreasing part
of the sequence.  Recall that the generating function for partitions
into distinct parts is $\prod_{j=1}^\infty (1+q^j)$.  The generating
function for $X$ is thus \emph{almost} given by the square of this
function, i.e.,
\[
\sum_{n \geq 0}\tilde{u}(n)q^n = \prod_{j=1}^\infty (1+q^j)^2.
\]
Unfortunately, the quantity $\tilde{u}(n)$ double-counts every
sequence of Type~1: once for the case where the
maximal element of the sequence comes from $\lambda$, and once for the
case where it comes from $\mu$.  It follows then that for $n \geq 1$,
we have
\begin{equation}\label{bounds}
\tilde{u}(n)/2 \leq v(n) \leq \tilde{u}(n).
\end{equation}
Of course $\tilde{u}(n)/2$ is an underestimate for $v(n)$, since,
although it corrects for the double-counting of Type~1
sequences, it halves the number of
Type~2 sequences.  However, the number of
Type~2 sequences is relatively small, and
consequently $v(n)$ is rather close to $\tilde{u}(n)/2$.

Let $c(n)$ be the number of words in $L$ of length $n$.  From
\eqref{bounds} and Theorem~\ref{bijection}, we have the following.

\begin{corollary}
For $n \geq 1$, the number $c(n)$ of binary words of length $n$ that
avoid $xx^Rx$ satisfies
\[
\tilde{u}(n) \leq c(n) \leq 2\tilde{u}(n).
\]
\end{corollary}

Rhoades \cite{Rho14} has recently determined the asymptotics of
$\tilde{u}(n)$; viz.,
\[
\tilde{u}(n) =
\frac{\sqrt{3}}{(24n-1)^{3/4}}\exp\left(\frac{\pi}{6}\sqrt{24n-1}\right)\left(1+\frac{(\pi^2-9)}{4\pi(24n-1)^{1/2}}+O\left(\frac{1}{n}\right)\right).
\]
This shows, as claimed, that the growth of $c(n)$ is intermediate
between polynomial and exponential.  The following table gives the
values of $c(n)$ and $\tilde{u}(n)$ for small values of $n$.

\begin{center}
\begin{tabular}{|c|*{13}{c}|}
\hline
$n$ & 0 & 1 & 2 & 3 & 4 & 5 & 6 & 7 & 8 & 9 & 10 & 11 & 12 \\ \hline
$c(n)$ & 1 & 2 & 4 & 6 & 10 & 16 & 24 & 34 & 50 & 72 & 100 & 138 & 188
  \\ \hline
$\tilde{u}(n)$ & 1 & 2 & 3 & 6 & 9 & 14 & 22 & 32 & 46 & 66 & 93 & 128
                                                      & 176 \\
\hline
\end{tabular}
\end{center}

The sequences $c(n)$ and $\tilde{u}(n)$ are sequences A261204 and
A022567 respectively of \emph{The Online Encyclopedia of Integer Sequences},
available online at \url{http://oeis.org}~.  We would also like to
note that the number $c(n)$ grows significantly faster than the number
of binary words of length $n$ that avoid $xxx^R$, whose order of
growth was previously estimated to be, roughly speaking, on the order
of $e^{\log^2 n}$ \cite{CR15}.

\section{Non-context-freeness of the language $L$}

Recall that in previous work we showed that the language $S$
of binary words avoiding $xxx^R$ has intermediate growth \cite{CR15}.
Adamczewski \cite{Ada15} observed that this implies that $S$ is not a
context-free language, since it is well known that context-free
languages have either polynomial or exponential growth.  He asked
if there is a ``direct proof'' of the non-context-freeness of $S$.
This seems to be difficult; we have not been able to come up with such
a proof.  Indeed, it even seems to be rather difficult to give a direct
proof that $S$ is not a regular language.

Adamczewski's observation applies just as well to the language $L$ of
binary words avoiding $xx^Rx$:  the intermediate growth shown above
implies that $L$ is not context-free.  However, unlike for $S$, it is
relatively easy to give a direct proof that $L$ is not context-free.
First, we observe that since the class of context-free languages is
closed under intersection with regular languages and under finite
transduction, if the language $L$ is context-free then the language
\[
L \cap (01)^+(10)^+(01)^+(10)^+ = \{(01)^i(10)^j(01)^k(10)^\ell : (i<j
\text{ or } k<j) \text{ and } (j<k \text{ or } \ell<k)\}
\]
is context-free, and in turn, the language
\[
\{a^ib^jc^kd^\ell :  (i<j \text{ or } k<j) \text{ and } (j<k \text{ or
} \ell<k)\}
\]
is context-free.  However, one can show using Ogden's Lemma that the
latter is not context-free.  Consequently $L$ is not context-free.

\section{Acknowledgments}
Both authors are supported by NSERC Discovery Grants.

\end{document}